\documentclass{amsart}
\usepackage{amsfonts,amssymb,amsmath,amsthm,dutchcal}
\usepackage{url}
\usepackage{enumerate}

\DeclareMathOperator{\dx}{dx}
\DeclareMathOperator{\dy}{dy}
\DeclareMathOperator{\dz}{dz}

\newcommand{\cu}{\mathcal{u}}
\newcommand{\cU}{\mathcal{U}}

\newcommand{\bbR}{\mathbb{R}}

\newtheorem{thrm}{Theorem}[section]
\newtheorem{lem}[thrm]{Lemma}

\theoremstyle{definition}
\newtheorem{definition}[thrm]{Definition}
\newtheorem{remark}[thrm]{Remark}
\numberwithin{equation}{section}

\author{Leobardo Rosales}
\address{Keimyung University \\
Department of Mathematics \\
1095 Dalgubeol-daero \\
Daegu, Republic of Korea, 42601}
\email{rosales.leobardo@gmail.com}
\thanks{This work was partly conducted by the author at the Korea Institute for Advanced Study, as an Associate Member.}

\begin{document}

\title[Hopf's lemma]{Generalizing Hopf's boundary point lemma}

\begin{abstract}
We give a Hopf boundary point lemma for weak solutions of linear divergence form uniformly elliptic equations, with H$\ddot{\text{o}}$lder continuous top-order coefficients and lower-order coefficients in a Morrey space. 
\end{abstract}
\maketitle

\section{Introduction} \label{introduction}

We illustrate how the Hopf boundary point lemma can be proved for divergence form equations, given sufficient regularity of the coefficients. Here, we show the case when the top-order coefficients are H$\ddot{\text{o}}$lder continuous, while the lower-order terms are in a \emph{Morrey space} (see Definition \ref{Morreyspace}).

The well-known Hopf boundary point lemma states that if $\cu \in C(B_{1}(0)) \cap C^{2}(\overline{B_{1}(0)})$ satisfies a second-order linear equation
$$\sum_{i,j=1}^{n} a^{ij} D_{i}D_{j} \cu + \sum_{i=1}^{n} c^{i}(x) D_{i} \cu + d \cu= 0$$
over $B_{1}(0),$ for functions $a^{ij}=a^{ji},c^{i},d \in L^{\infty}(B_{1}(0))$ for $i,j \in \{ 1,\ldots,n \}$ with $\{ a^{ij} \}_{i,j=1}^{n}$ \emph{uniformly elliptic} over $B_{1}(0)$ with respect to some $\lambda \in (0,\infty)$ (see Definition \ref{uniformlyelliptic}), and if $\cu(x)>\cu(-e_{n})=0$ for all $x \in B_{1}(0),$ then
\begin{equation} \label{liminf}
\liminf_{h \searrow 0} \frac{\cu((h-1)e_{n})}{h} > 0.
\end{equation}
See the proof given by Hopf in \cite{H52}, as well as Lemma 3.4 of \cite{GT83}.

It is useful to have the Hopf boundary point lemma for divergence form equations. We consider $\cu \in C(\overline{B_{1}(0)}) \cap W^{1,2}(B_{1}(0))$ a weak solution over $B_{1}(0)$ of the equation
\begin{equation} \label{equation}
\sum_{i,j=1}^{n} D_{i} \left( a^{ij} D_{j}\cu +b^{i} \cu \right) + \sum_{i=1}^{n} c^{i} D_{i}\cu + d\cu = 0
\end{equation}
for functions $a^{ij},c^{i} \in L^{2}(B_{1}(0))$ and $b^{i},d \in L^{1}(B_{1}(0))$ for each $i,j \in \{ 1,\ldots,n \}$ (see Definition \ref{weaksolution}). Assuming again that $\cu(x) > \cu(-e_{n})=0$ for all $x \in B_{1}(0),$ the aim is to show \eqref{liminf} holds.

The most recent result is given by Theorem 1.1 of \cite{S15}, which shows \eqref{liminf} holds if $a^{ij}=a^{ji} \in C^{0,\alpha}(\overline{B_{1}(0)})$ for some $\alpha \in (0,1)$ for $i,j \in \{ 1,\ldots,n \}$ with $\{ a^{ij} \}_{i,j=1}^{n}$ uniformly elliptic with respect to some $\lambda \in (0,\infty)$ (see Definition \ref{uniformlyelliptic}), and $b^{i}=0$ while $c^{i},d \in L^{\infty}(B_{1}(0))$ for each $i \in \{ 1,\ldots,n \}.$ We as well refer the reader to \cite{S15} which discusses previous generalizations of the Hopf boundary point lemma, and gives examples showing the assumption $a^{ij} \in C^{0,\alpha}(\overline{B_{1}(0)})$ for each $i,j \in \{ 1,\ldots,n \}$ cannot be relaxed.

Here, we prove in Theorem \ref{hopflemma} that \eqref{liminf} holds under the more general assumption that the coefficients in \eqref{equation} satisfy for each $i,j \in \{ 1,\ldots,n \}$
$$\begin{array}{ccc}
a^{ij},b^{i} \in C^{0,\alpha}(\overline{B_{1}(0)}), & c^{i} \in L^{q}(B_{1}(0)), & d \in L^{\frac{q}{2}}(B_{1}(0)) \cap L^{1,\alpha}(B_{1}(0))
\end{array}$$
for some $q>n$ and $\alpha \in (0,1)$; see Remark \ref{hopfremark}(i). We also assume $a^{ij}(-e_{n})=a^{ji}(-e_{n})$ for each $i,j \in \{ 1,\ldots,n \}$ with $\{ a^{ij} \}_{i,j=1}^{n}$ uniformly elliptic over $B_{1}(0)$ with respect to some $\lambda \in (0,\infty).$ Additionally, we assume $\{ b^{i} \}_{i=1}^{n},d$ are \emph{weakly non-positive} over $B_{1}(0)$ (see Definition \ref{weaklynonpositive}).

The space $L^{1,\alpha}(B_{1}(0))$ denotes a \emph{Morrey space} (see Definition \ref{Morreyspace}). Morrey spaces were introduced in \cite{M66} to study the existence and regularity of solutions to elliptic systems. Consequentially, to prove Theorem \ref{hopflemma} we must use the $C^{1,\alpha}$ estimate of Theorem 5.5.5'(b) of \cite{M66}, stated here for convenience as Lemma \ref{Morreyestimate}.

Since their introduction, Morrey spaces have been studied in and outside the study of partial differential equations. Recent work has been done in the study of elliptic and parabolic partial differential equations involving data in the $L^{1,\alpha}$ Morrey space. We refer to the seminal work in this direction given by \cite{M07}, which uses Morrey spaces to prove regularity results for solutions to non-linear divergence-form elliptic equations having inhomogeneous term a measure. To see further, recent work resulting from and related to \cite{M07}, using $L^{1,\alpha}$ Morrey spaces to study elliptic and parabolic equations in various settings, we refer the reader to the works: \cite{AKM17}, \cite{CCAL17}, \cite{CAL17}, \cite{CL06}, \cite{CL10}, \cite{CL14}, \cite{CLS08}, \cite{KM10}, \cite{L11}, \cite{L14}, \cite{LS10}, \cite{LS11}, \cite{LS12}, \cite{LS15}.

Our underlying goal is to illustrate how the Hopf boundary point lemma can be shown in other settings for divergence form equations. To this end, the proof of Theorem \ref{hopflemma} is given in five steps demonstrating the necessary theoretical ingredients. The structure of the proof is taken from the proof of Lemma 10.1 of \cite{HS79}, which shows one generalization of the Hopf boundary point lemma to divergence form equations.

We only assume working knowledge of real analysis and ready \emph{access} to the reference \cite{GT83}. Otherwise, the crucial estimate Theorem 5.5.5'(b) we carefully state in the present setting as Lemma \ref{Morreyestimate} in Section \ref{estimateandexistencelemmas}. We begin by stating in Section \ref{preliminaries} our basic definitions and some preliminary calculations needed in Section \ref{estimateandexistencelemmas} to prove the necessary existence result Lemma \ref{Morreyexistence}. We also state in Section \ref{estimateandexistencelemmas} the weak maximum principle needed, Lemma \ref{weakmaximumprinciple}. In Section \ref{thehopfboundarypointlemma} we prove the Hopf boundary point lemma, Theorem \ref{hopflemma}.

\section{Preliminaries} \label{preliminaries}

We will work in $\bbR^{n}$ with $n \geq 2.$ We denote the volume of the open unit ball $B_{1}(0) \subset \bbR^{n}$ by $\omega_{n} = \int_{B_{1}(0)} \dx.$ Standard notation for the various spaces of functions shall be used; in particular $C^{1}_{c}(\cU,[0,\infty))$ shall denote the set of non-negative continuously differentiable functions with compact support in an open set $\cU \subseteq \bbR^{n}.$

We begin by giving the definition of a family of Morrey spaces, to which we will relax the assumptions on the lower-order terms given in \cite{S15}.

\begin{definition} \label{Morreyspace}
Suppose $\alpha \in (0,1)$ and $\cU \subseteq \bbR^{n}$ is an open set. We say $d \in L^{1,\alpha}(\cU)$ if $d \in L^{1}(\cU)$ with finite $L^{1,\alpha}(\cU)$ norm, defined by
$$\| d \|_{L^{1,\alpha}(\cU)} := \sup_{x \in \bbR^{n}, \rho \in (0,\infty)}  \frac{1}{\rho^{n-1+\alpha}} \int_{\cU \cap B_{\rho}(x)} |d(y)| \dy.$$
\end{definition}

\begin{remark} \label{Morreyspaceremark}
If $q>n,$ $\cU \subset \bbR^{n}$ is a bounded open set, and $c \in L^{q}(\cU),$ then $c \in L^{1,\alpha}(\cU)$ for $\alpha = 1-\frac{n}{q} \in (0,1),$ with
$$\| c \|_{L^{1,\alpha}(\cU)} \leq \left( \int_{\cU} \dx \right)^{1-\frac{1}{q}} \| c \|_{L^{q}(\cU)}.$$
\end{remark}

Next, we state what it means for $\cu$ to be a weak \emph{supersolution} (respectively \emph{solution}, \emph{subsolution}) to a linear divergence form equation. The assumptions on the coefficients are to ensure integrability.

\begin{definition} \label{weaksolution}
Let $\cU \subset \bbR^{n}$ be an open set, and suppose $a^{ij},c^{i} \in L^{2}(\cU),$ $b^{i},d,g,f^{i} \in L^{1}(\cU)$ for each $i,j \in \{ 1,\ldots,n\}.$ We say $\cu \in L^{\infty}(\cU) \cap W^{1,2}(\cU)$ \emph{is a weak solution over $\cU$ of the equation}
$$\sum_{i,j=1}^{n} D_{i} \left( a^{ij} D_{j}\cu + b^{i}\cu \right) + \sum_{i=1}^{n} c^{i} D_{i}\cu + d\cu \leq g+\sum_{i=1}^{n} D_{i} f^{i}$$
(respectively $\geq,=$) if for all $\zeta \in C^{1}_{c}(\cU;[0,\infty))$ 
$$\begin{aligned}
\int \sum_{i,j=1}^{n} a^{ij} D_{j} \cu D_{i} \zeta + \sum_{i=1}^{n} & \left( b^{i} \cu D_{i} \zeta - c^{i} (D_{i} \cu) \zeta \right) - d \cu \zeta \dx \\
& \geq \int -g \zeta + \sum_{i=1}^{n} f^{i} D_{i} \zeta \dx
\end{aligned}$$
(respectively $\leq,=$).
\end{definition}

The next two definitions should be regarded as holding throughout. 

\begin{definition} \label{uniformlyelliptic}
Let $\lambda \in (0,\infty),$ $\cU \subseteq \bbR^{n},$ and suppose we have functions $a^{ij}: \cU \rightarrow \bbR$ for each $i,j \in \{ 1,\ldots,n \}.$ We say $\{ a^{ij} \}_{i,j=1}^{n}$ are \emph{uniformly elliptic over $\cU$ with respect to $\lambda$} if
$$\sum_{i,j=1}^{n} a^{ij}(x) \xi_{i} \xi_{j} \geq \lambda |\xi|^{2} \text{ for each } x \in \cU \text{ and } \xi \in \bbR^{n}.$$
\end{definition}

\begin{definition} \label{weaklynonpositive}
Let $\cU \subseteq \bbR^{n}$ be an open set, and suppose $b^{i},d \in L^{1}(\cU)$ for each $i \in \{ 1,\ldots,n \}.$ We say $\{ b^{i} \}_{i=1}^{n},d$ are \emph{weakly non-positive over $\cU$} if
$$\int d \zeta - \sum_{i=1}^{n} b^{i} D_{i} \zeta \dx \leq 0$$
for each $\zeta \in C^{1}_{c}(\cU,[0,\infty)).$
\end{definition}

Proving Theorem \ref{hopflemma} will require the existence result Lemma \ref{Morreyestimate}, which in turn we prove using a well-known existence result given by Theorem 8.34 of \cite{GT83}.We must thus discuss mollification and Morrey spaces. 

\begin{definition} \label{mollification}
Let $\Omega = B_{1}(0) \setminus \overline{B_{\frac{1}{2}}(0)},$ and fix $v \in C^{\infty}_{c}(B_{1}(0);[0,\infty))$ a standard mollifier. For $\delta \in (0,\frac{1}{4})$ denote $v_{\delta}(x) = \frac{1}{\delta^{n}} v(\frac{x}{\delta})$ and define $\gamma_{\delta}: \bbR^{n} \setminus \{ 0 \} \rightarrow \bbR^{n}$ by
$$\gamma_{\delta}(x) = ((1-4\delta)|x| + 3 \delta) \frac{x}{|x|} = (1-4 \delta)x + 3 \delta  \frac{x}{|x|}.$$
Using these functions, we make the following definitions.
\begin{enumerate}
 \item[(i)] Given $d \in L^{1}(\Omega),$ we extend $d(y)=0$ for $y \in \bbR^{n} \setminus \Omega$ and define the usual convolution $d \ast v_{\delta}: \bbR^{n} \rightarrow \bbR.$ We also define the weighted convolution $d \circledast v_{\delta}: \bbR^{n} \setminus \{0\} \rightarrow \bbR$ by $d \circledast v_{\delta} = J \gamma_{\delta} \big( (d \ast v_{\delta}) \circ \gamma_{\delta} \big)$ where $J\gamma_{\delta} = |\det(D \gamma_{\delta})|.$
 \item[(ii)] Given $\{ b^{i} \}_{i=1}^{n} \subset C^{0,\alpha}(\overline{\Omega}),$ define $b^{i} \star v_{\delta}: \overline{\Omega} \rightarrow \bbR$ for $i \in \{ 1,\ldots,n \}$ by
$$b^{i} \star v_{\delta} = \sum_{j=1}^{n} \Big( (D_{j}( e_{i} \cdot \gamma^{-1}_{\delta})) \circ \gamma_{\delta} \Big) \cdot (b^{j} \circledast v_{\delta}).$$
\end{enumerate}
\end{definition}

We will use these convolutions to prove the existence result Lemma \ref{Morreyexistence}. For this, we need the following calculations.

\begin{lem} \label{mollificationproperties}
Denote $\Omega = B_{1}(0) \setminus \overline{B_{\frac{1}{2}}(0)}.$ Suppose $b^{i} \in C^{0,\alpha}(\overline{\Omega})$ for $i \in \{ 1,\ldots,n \}$ and $d \in L^{1,\alpha}(\Omega)$ with $\alpha \in (0,1).$
\begin{enumerate}
 \item[(i)] For $\delta \in (0,\frac{1}{8})$ the convolutions satisfy
\begin{itemize}
 \item $d \ast v_{\delta} \in C^{\infty}(\bbR^{n}),$ $\| d \ast v_{\delta} \|_{L^{1,\alpha}(\Omega)} \leq \| d \|_{L^{1,\alpha}(\Omega)},$ \\ and $d \ast v_{\delta} \rightarrow d \text{ in } L^{1}(\Omega)$ as $\delta \searrow 0.$
 \item $d \circledast v_{\delta} \in C^{\infty}(\bbR^{n} \setminus \{0\}),$ $\| d \circledast v_{\delta} \|_{L^{1,\alpha}(\Omega)} \leq 2^{n} \| d \|_{L^{1,\alpha}(\Omega)},$ \\ and $d \circledast v_{\delta} \rightarrow d \text{ in } L^{1}(\Omega)$ as $\delta \searrow 0.$ 
 \item There exists $C_{\ref{mollificationproperties}} = C_{\ref{mollificationproperties}}(n) \in (0,\infty)$ so that \\ $b^{i} \star v_{\delta} \in C^{\infty}(\overline{\Omega}),$ $\| b^{i} \star v_{\delta} \|_{C^{0,\alpha}(\overline{\Omega})} \leq C_{\ref{mollificationproperties}} \sum_{j=1}^{n} \| b^{j} \|_{C^{0,\alpha}(\overline{\Omega})},$ \\ and $b^{i} \star v_{\delta} \rightarrow b^{i}$ in $L^{1}(\Omega)$ as $\delta \searrow 0$ for each $i \in \{ 1,\ldots,n \}.$
\end{itemize}
 \item[(ii)] If $\{ b^{i} \}_{i=1}^{n},d$ are weakly non-positive over $\Omega,$ then $\{ b^{i} \star v_{\delta} \}_{i=1}^{n},d \circledast v_{\delta}$ are weakly non-positive over $\Omega.$
\end{enumerate}
\end{lem}

\begin{proof}
We leave some details to the reader, which follow from standard real analysis.

First, consider (i). We discuss each item separately. 
\begin{itemize}
 \item Since $v \in C^{\infty}_{c}(B_{1}(0);[0,\infty))$ is a standard mollifier, then $d \ast v_{\delta} \in C^{\infty}(\bbR^{n})$ and $d \ast v_{\delta} \rightarrow d$ in $L^{1}(\Omega)$ as $\delta \searrow 0$ are well-known real analysis facts.

 Next, fix $x \in \bbR^{n}$ and $\rho \in (0,\infty).$ We compute using the definition of the convolution, Fubini's theorem, a change of variables and the extension $d(y)=0$ for $y\in \bbR^{n} \setminus \Omega,$ the definition of the $L^{1,\alpha}$ norm, and $\int v_{\delta}(z) \dz = 1$ since $v \in C^{\infty}_{c}(B_{1}(0);[0,\infty))$ is a standard mollifier
 $$\begin{aligned}
 \frac{1}{\rho^{n-1+\alpha}} \int_{\Omega \cap B_{\rho}(x)} |d \ast v_{\delta}(y)| \dy \leq & \frac{1}{\rho^{n-1+\alpha}} \int_{\Omega \cap B_{\rho}(x)} \int |d(y-z)| v_{\delta}(z) \dz \dy \\
 \leq & \int \frac{1}{\rho^{n-1+\alpha}} \int_{\Omega \cap B_{\rho}(x)} |d(y-z)| \dy v_{\delta}(z) \dz \\
 \leq & \int \frac{1}{\rho^{n-1+\alpha}} \int_{\Omega \cap B_{\rho}(x-z)} |d(y)| \dy v_{\delta}(z) \dz \\
 \leq & \int \| d \|_{L^{1,\alpha}(\Omega)} v_{\delta}(z) \dz = \| d \|_{L^{1,\alpha}(\Omega)}.
 \end{aligned}$$
 This verifies $\| d \ast v_{\delta} \|_{L^{1,\alpha}(\Omega)} \leq \| d \|_{L^{1,\alpha}(\Omega)}.$
 \item Observe that $\gamma_{\delta} \in C^{\infty}(\bbR^{n} \setminus \{ 0\})$ and $\gamma_{\delta}$ converges smoothly over $\overline{\Omega}$ to the identity as $\delta \searrow 0.$ Using these facts, Definition \ref{mollification}(i), and the previous item we can show $d \circledast v_{\delta} \in C^{\infty}(\bbR^{n} \setminus \{0\})$ and $d \circledast v_{\delta} \rightarrow d \text{ in } L^{1}(\Omega)$ as $\delta \searrow 0.$

 Fix any $x \in \bbR^{n}$ and $\rho \in (0,\infty).$ By Definition \ref{mollification} we can check
 $$\gamma_{\delta}(\Omega) = B_{1-\delta}(0) \setminus \overline{B_{\frac{1}{2}}(0)} \text{ and } \gamma_{\delta}(B_{\rho}(x)) \subseteq B_{(1+8\delta)\rho}(\gamma_{\delta}(z)).$$ 
 Using this together with Definition \ref{mollification}(i), $\gamma_{\delta}$ as a change of variables, $\delta \in (0,\frac{1}{8}),$ $\alpha \in (0,1),$ and the previous item we compute
  $$\begin{aligned}
 \frac{1}{\rho^{n-1+\alpha}} \int_{\Omega \cap B_{\rho}(x)} |d \circledast v_{\delta}(y)| \dy = & \frac{1}{\rho^{n-1+\alpha}} \int_{\gamma_{\delta}(\Omega \cap B_{\rho}(x))} |d \ast v_{\delta}(y)| \dy \\
 \leq &  \frac{1}{\rho^{n-1+\alpha}} \int_{\Omega \cap B_{2\rho}(\gamma_{\delta}(x))} |d \ast v_{\delta}(y)| \dy \\
 \leq & 2^{n-1+\alpha} \| d \|_{L^{1,\alpha}(\Omega)} \leq 2^{n} \| d \|_{L^{1,\alpha}(\Omega)}.
 \end{aligned}$$
 We conclude $ \| d \circledast v_{\delta} \|_{L^{1,\alpha}(\Omega)} \leq 2^{n} \| d \|_{L^{1,\alpha}(\Omega)}.$
 \item Observe that $\gamma_{\delta}: \overline{\Omega} \rightarrow (\overline{B_{1-\delta}(0)} \setminus B_{\frac{1}{2}+\delta}(0))$ is invertible with 
 $$\gamma_{\delta}^{-1} \in C^{\infty}(\overline{B_{1-\delta}(0)} \setminus B_{\frac{1}{2}+\delta}(0);\overline{\Omega});$$
 since $\gamma_{\delta}$ converges smoothly over $\overline{\Omega}$ to the identity as $\delta \searrow 0,$ we conclude
 $$\Big( (D_{j}( e_{i} \cdot \gamma^{-1}_{\delta})) \circ \gamma_{\delta} \Big) \rightarrow \left\{ \begin{array}{ll} 1 & \text{if } j=i \\ 0 & \text{if } j \neq i \end{array} \right.$$
 uniformly over $\overline{\Omega}$ as $\delta \searrow 0,$ for each $i,j \in \{ 1,\ldots,n \}.$ These facts together with Definition \ref{mollification}(ii) and the previous item applied to $b^{j} \circledast v_{\delta}$ for each $j \in \{ 1,\ldots,n \},$ we can show $b^{i} \star v_{\delta} \in C^{\infty}(\overline{\Omega})$ and $b^{i} \star v_{\delta} \rightarrow b^{i}$ in $L^{1}(\Omega)$ as $\delta \searrow 0$ for each $i \in \{ 1,\ldots,n \}.$

Next, using the definition of the $C^{0,\alpha}$ norm and Definition \ref{mollification}(ii) we compute for each $i \in \{ 1,\ldots,n \}$
$$\begin{aligned}
\| & b^{i} \star v_{\delta} \|_{C^{0,\alpha}(\overline{\Omega})} \\
& \leq \sum_{j=1}^{n} \| (D_{j}(e_{i} \cdot \gamma_{\delta}^{-1})) \circ \gamma_{\delta} \|_{C^{0,\alpha}(\overline{\Omega})} \| J\gamma_{\delta} \|_{C^{0,\alpha}(\overline{\Omega})} \| (b^{j} \ast v_{\delta}) \circ \gamma_{\delta} \|_{C^{0,\alpha}(\overline{\Omega})}.
\end{aligned}$$
We also compute again using the definition of the $C^{0,\alpha}$ norm
$$\begin{aligned}
\| & (b^{j} \ast v_{\delta}) \circ \gamma_{\delta} \|_{C^{0,\alpha}(\overline{\Omega})} \\
\leq & \max \left\{ 1,\sup_{x,y \in \overline{\Omega},x \neq y} \frac{|\gamma_{\delta}(x)-\gamma_{\delta}(y)|^{\alpha}}{|x-y|^{\alpha}} \right\} \| b^{j} \ast v_{\delta} \|_{C^{0,\alpha}(\overline{B_{1-\delta}(0)} \setminus B_{\frac{1}{2}+\delta}(0))}
\end{aligned}$$
for each $j \in \{ 1,\ldots,n \}.$ Since $\gamma_{\delta}(x) = (1-4\delta)x+3\delta \frac{x}{|x|}$ for $x \in \bbR^{n} \setminus \{ 0 \},$ we can find $C_{\ref{mollificationproperties}}=C_{\ref{mollificationproperties}}(n) \in (0,\infty)$ so that for each $\delta \in (0,\frac{1}{8})$ we have
$$\left. \begin{aligned}
& \| (D_{j}(e_{i} \cdot \gamma_{\delta}^{-1})) \circ \gamma_{\delta} \|_{C^{0,\alpha}(\overline{\Omega})} \| J\gamma_{\delta} \|_{C^{0,\alpha}(\overline{\Omega})} \\
& \times \max \left\{ 1,\sup_{x,y \in \overline{\Omega},x \neq y} \frac{|\gamma_{\delta}(x)-\gamma_{\delta}(y)|^{\alpha}}{|x-y|^{\alpha}} \right\}
\end{aligned} \right\} \leq C_{\ref{mollificationproperties}}$$
for each $i,j \in \{ 1,\ldots,n \}.$ These three calculations taken together imply
$$\| b^{i} \star v_{\delta} \|_{C^{0,\alpha}(\overline{\Omega})} \leq C_{\ref{mollificationproperties}} \sum_{j=1}^{n}  \| b^{j} \ast v_{\delta} \|_{C^{0,\alpha}(\overline{B_{1-\delta}(0)} \setminus B_{\frac{1}{2}+\delta}(0))}$$
for each $i \in \{ 1,\ldots,n \}.$ To conclude $\| b^{i} \star v_{\delta} \|_{C^{0,\alpha}(\overline{\Omega})} \leq C_{\ref{mollificationproperties}} \sum_{j=1}^{n} \| b^{j} \|_{C^{0,\alpha}(\overline{\Omega})}$ as needed, it therefore suffices to verify
$$\| b^{j} \ast v_{\delta} \|_{C^{0,\alpha}(\overline{B_{1-\delta}(0)} \setminus B_{\frac{1}{2}+\delta}(0))} \leq \| b^{j} \|_{C^{0,\alpha}(\overline{\Omega})}.$$
This follows from the fact that $v \in C^{\infty}_{c}(B_{1}(0);[0,\infty))$ is a standard mollifier and the definitions of the convolution and the $C^{0,\alpha}$ norm. For example, given $x,y \in \overline{B_{1-\delta}(0)} \setminus B_{\frac{1}{2}+\delta}(0)$ we can compute
$$\begin{aligned}
|b^{j} \ast v_{\delta}(x)-b^{j} \ast v_{\delta}(y)| \leq & \int |b^{j}(x+\delta z)-b^{j}(y+ \delta z)| v(z) \dz \\
\leq & |x-y|^{\alpha} \| b^{j} \|_{C^{1,\alpha}(\overline{\Omega})}.
\end{aligned}$$
We leave the details to the reader, and conclude the required estimate $\| b^{i} \star v_{\delta} \|_{C^{0,\alpha}(\overline{\Omega})} \leq C_{\ref{mollificationproperties}} \sum_{j=1}^{n} \| b^{j} \|_{C^{0,\alpha}(\overline{\Omega})}$ for each $i \in \{ 1,\ldots, n \}.$
\end{itemize}

Next, consider (ii). Proving $\{ b^{i} \star v_{\delta} \}_{i=1}^{n},d \circledast v_{\delta}$ are weakly non-positive over $\Omega$ is done in two steps. First, we check using the definition of the convolution that $\{ b^{i} \ast v_{\delta} \}_{i=1}^{n},d \ast v_{\delta}$ are weakly non-positive over $B_{1-\delta}(0) \setminus \overline{B_{\frac{1}{2}+\delta}(0)};$ we leave this to the reader. Second, using $\gamma_{\delta}: \Omega \rightarrow B_{1-\delta}(0) \setminus \overline{B_{\frac{1}{2}+\delta}(0)}$ as a change of variables we can check for $\zeta \in C^{1}_{c}(\Omega;[0,\infty))$
$$\begin{aligned}
\int & (d \circledast v_{\delta}) \zeta - \sum_{i=1}^{n} (b^{i} \star v_{\delta}) D_{i} \zeta \dx \\
& = \int (d \ast v_{\delta}) (\zeta \circ \gamma_{\delta}^{-1}) - \sum_{i,j=1}^{n} (b^{j} \ast v_{\delta})(D_{j}(e_{i} \cdot \gamma_{\delta}^{-1}))((D_{i} \zeta) \circ \gamma_{\delta}^{-1}) \dx \\
& = \int (d \ast v_{\delta}) (\zeta \circ \gamma_{\delta}^{-1}) - \sum_{j=1}^{n} (b^{j} \ast v_{\delta}) D_{j}(\zeta \circ \gamma_{\delta}^{-1}) \dx \leq 0,
\end{aligned}$$ 
since $\{ b^{j} \ast v_{\delta} \}_{j=1}^{n},d \ast v_{\delta}$ are weakly non-positive over $B_{1-\delta}(0) \setminus \overline{B_{\frac{1}{2}+\delta}(0)};$
\end{proof}

\subsection{Acknowledgements}

This work was partly conducted by the author at the Korea Institute for Advanced Study, as an Associate Member. 

\section{Estimate and Existence Lemmas} \label{estimateandexistencelemmas}

In this section we prove the necessary a priori gradient estimate, existence, and weak maximum principle results needed to prove Theorem \ref{hopflemma}. 

\begin{lem}[Morrey estimate] \label{Morreyestimate}
Suppose $\lambda,J \in (0,\infty),$ $\alpha \in (0,1),$ and let $\Omega = B_{1}(0) \setminus \overline{B_{\frac{1}{2}}(0)}.$ There is $C_{\ref{Morreyestimate}}=C_{\ref{Morreyestimate}}(n,\lambda,J,\alpha) \in (0,\infty)$ so that if 
\begin{enumerate}
 \item[(i)] $a^{ij},b^{i} \in C^{0,\alpha}(\overline{\Omega})$ and $c^{i},d \in L^{1,\alpha}(\Omega)$ for $i,j \in \{ 1,\ldots,n \},$ 
 \item[(i)] $\{a^{ij}\}_{i,j=1}^{n}$ are uniformly elliptic over $\Omega$ with respect to $\lambda,$
 \item[(ii)] $\sum_{i,j}^{n} \| a^{ij} \|_{C^{0,\alpha}(\overline{\Omega})} + \sum_{i=1}^{n} \left( \| b^{i} \|_{C^{0,\alpha}(\overline{\Omega})} + \| c^{i} \|_{L^{1,\alpha}(\Omega)} \right) + \| d \|_{L^{1,\alpha}(\Omega)}  \leq J,$
\end{enumerate}
and if $\cu \in C^{1,\alpha}(\overline{\Omega})$ is a weak solution over $\Omega$ of the equation
$$\sum_{i,j=1}^{n} D_{i} \left( a^{ij} D_{j}\cu + b^{i}\cu \right) + \sum_{i=1}^{n} c^{i} D_{i}\cu + d\cu = g+\sum_{i=1}^{n} D_{i}f^{i}$$
with $g \in L^{1,\alpha}(\Omega)$ and $f \in C^{0,\alpha}(\overline{\Omega}),$ then
$$\| \cu \|_{C^{1,\alpha}(\overline{\Omega})} \leq C_{\ref{Morreyestimate}} \left( \| \cu \|_{L^{1}(\Omega)} + \| g \|_{L^{1,\alpha}(\Omega)} + \sum_{i=1}^{n} \| f^{i} \|_{ C^{0,\alpha}(\overline{\Omega})} \right).$$
\end{lem}

\begin{proof}
This is Theorem 5.5.5'(b) of \cite{M66}
$$\text{(with $\mu,G,e,f$ replaced respectively by $\alpha,\Omega,\{ f^{i} \}_{i=1}^{n},g$)}.$$
The $C^{1,\alpha}$-conditions (that is, the ``$C^{1}_{\mu}$-conditions'' as stated in Definition 5.5.2 of \cite{M66}) are implied by (i). To more clearly see the dependence $C_{\ref{Morreyestimate}}=C_{\ref{Morreyestimate}}(n,\lambda,J,\alpha),$ see Theorem 5.5.2(b) of \cite{M66}.
\end{proof}

Next, we state for convenience the more general version of Theorem 8.16 of \cite{GT83}, using the remark on page 193 of \cite{GT83}.

\begin{lem}[Weak maximum principle] \label{weakmaximumprinciple}
Suppose $q>n$ and $\lambda,k \in (0,\infty).$ Denote $\Omega = B_{1}(0) \setminus \overline{B_{\frac{1}{2}}(0)}.$ There is $C_{\ref{weakmaximumprinciple}}=C_{\ref{weakmaximumprinciple}}(n,q,\lambda,k) \in (0,\infty)$ so that if
\begin{enumerate}
 \item[(i)] $a^{ij} \in L^{\infty}(\Omega),$ $b^{i},c^{i} \in L^{q}(\Omega)$ for $i,j \in \{ 1,\ldots,n \}$ and $d \in L^{\frac{q}{2}}(\Omega),$
 \item[(ii)] $\{a^{ij}\}_{i,j=1}^{n}$ are uniformly elliptic over $\Omega$ with respect to $\lambda,$
 \item[(iii)] $\{ b^{i} \}_{i=1}^{n},d$ are weakly non-positive over $\Omega,$
 \item[(iv)] $\sum_{i=1}^{n} \left( \| b^{i} \|_{L^{q}(\Omega)} + \| c^{i} \|_{L^{q}(\Omega)} \right) + \| d \|_{L^{\frac{q}{2}}(\Omega)}  \leq k,$
\end{enumerate}
and if $\cu \in C(\overline{\Omega}) \cap W^{1,2}(\Omega)$ is a solution over $\Omega$ of the equation 
$$\sum_{i,j=1}^{n} D_{i} \left( a^{ij} D_{j}\cu + b^{i}\cu \right) + \sum_{i=1}^{n} c^{i} D_{i}\cu + d\cu \leq g+\sum_{i=1}^{n} D_{i}f^{i}$$
with $g \in L^{\frac{q}{2}}(\Omega)$ and $f^{i} \in L^{q}(\Omega)$ for each $i \in \{ 1,\ldots,n \},$ then
$$\inf_{\Omega} \cu \geq \inf_{\partial \Omega} \min \{ 0,\cu \} - C_{\ref{weakmaximumprinciple}} \left( \| g \|_{L^{\frac{q}{2}}(\Omega)} + \sum_{i=1}^{n} \| f^{i} \|_{L^{q}(\Omega)} \right).$$
\end{lem}

We use Lemmas \ref{Morreyestimate},\ref{weakmaximumprinciple} to show we can solve linear divergence form equations with lower-order terms in a Morrey space. This will allow us to get the barrier functions in step 2 of the proof of Theorem \ref{hopflemma}.

\begin{lem} \label{Morreyexistence}
Suppose $q>n$ and $\lambda \in (0,\infty).$ Denote $\alpha = 1-\frac{n}{q}$ and $\Omega = B_{1}(0) \setminus \overline{B_{\frac{1}{2}}(0)}.$ Also suppose we have functions
\begin{enumerate}
 \item[(i)] $ a^{ij},b^{i} \in C^{0,\alpha}(\overline{\Omega}),$ $c^{i} \in L^{q}(\Omega)$ for $i,j \in \{ 1,\ldots,n \}$ and $d \in L^{\frac{q}{2}}(\Omega) \cap L^{1,\alpha}(\Omega),$
 \item[(ii)] $\{a^{ij}\}_{i,j=1}^{n}$ are uniformly elliptic over $\Omega$ with respect to $\lambda,$ 
 \item[(iii)] $\{ b^{i} \}_{i=1}^{n},d$ are weakly non-positive over $\Omega.$
\end{enumerate}
Then there is $\varphi \in C^{1,\alpha}(\overline{\Omega})$ which is a weak solution over $\Omega$ of the equation
$$\sum_{i,j=1}^{n} D_{i} \left( a^{ij} D_{j}\varphi + b^{i}\varphi \right) + \sum_{i=1}^{n} c^{i}D_{i}\varphi + d\varphi = 0$$
with $\varphi(x) = \left\{ \begin{array}{cc} 0 & \text{for } x \in \partial B_{1}(0) \\ -1 & \text{for } x \in \partial B_{\frac{1}{2}}(0) \end{array} \right.$ and $\varphi(x) \in [-1,0] \text{ for each } x \in \overline{\Omega}.$
\end{lem}

\begin{proof}
We follow the proof of and use directly Theorem 8.34 on page 211 of \cite{GT83}.

Define for $\delta \in (0,\frac{1}{8})$ and each $i \in \{ 1,\ldots,n \}$
\begin{equation} \label{memollification}
\begin{array}{ccc}
b^{i}_{\delta} = b^{i} \star v_{\delta}, &
c^{i}_{\delta} = c^{i} \ast v_{\delta}, &
d_{\delta} = d \circledast v_{\delta}
\end{array}
\end{equation}
by Definition \ref{mollification}. Now consider the weakly defined operator over $\Omega$
$$L_{\delta}u = \sum_{i,j=1}^{n} D_{i} \left( a^{ij}D_{j}u + b^{i}_{\delta} u \right) + \sum_{i=1}^{n} c^{i}_{\delta} D_{i}u + d_{\delta}u.$$
Then (i),(ii),(iii), \eqref{memollification}, and Lemma \ref{mollificationproperties} imply $L_{\delta}$ satisfies (8.5),(8.8),(8.85) of \cite{GT83} over $\Omega$ with
$$K = \sum_{i,j=1}^{n} \| a^{ij} \|_{C^{0,\alpha}(\overline{\Omega})} + \sum_{i=1}^{n} \left( \| b^{i}_{\delta} \|_{C^{0,\alpha}(\overline{\Omega})} + \| c^{i}_{\delta} \|_{L^{\infty}(\Omega)} \right) + \| d_{\delta} \|_{L^{\infty}(\Omega)}.$$
We can thus apply Theorem 8.34 of \cite{GT83} 
$$\text{(with $b^{i},c^{i},d,g,f^{i}$ replaced respectively by $b^{i}_{\delta},c^{i}_{\delta},d_{\delta},0,0$)}$$
over $\Omega = B_{1}(0) \setminus \overline{B_{\frac{1}{2}}(0)}$ with operator $L_{\delta}$ to conclude the generalized Dirichlet problem
$$\begin{array}{lr}
L_{\delta}u = 0 \text{ in } \Omega, & u = \left\{ \begin{array}{cc} 0 & \text{over } \partial B_{1}(0) \\ -1 & \text{over } \partial B_{\frac{1}{2}}(0) \end{array} \right.
\end{array}$$
is uniquely solvable in $C^{1,\alpha}(\Omega).$ Letting $\varphi_{\delta} \in C^{1,\alpha}(\overline{\Omega})$ be this unique solution, and comparing (8.2) of \cite{GT83} with Definition \ref{weaksolution}, we conclude $\varphi_{\delta}$ is a weak solution over $\Omega$ of the equation
\begin{equation} \label{varphidelta}
\begin{aligned}
& \sum_{i,j=1}^{n} D_{i} \left( a^{ij}D_{j} \varphi_{\delta} + b^{i}_{\delta} \varphi_{\delta} \right) + \sum_{i=1}^{n} c^{i}_{\delta} D_{i} \varphi_{\delta} + d_{\delta} \varphi_{\delta} = 0 \\
& \text{with } \varphi_{\delta}(x) = \left\{ \begin{array}{cc} 0 & \text{for } x \in \partial B_{1}(0) \\ -1 & \text{for } x \in \partial B_{\frac{1}{2}}(0) \end{array}. \right.
\end{aligned}
\end{equation}
We also apply Lemma \ref{weakmaximumprinciple} (with $f^{i},g=0$ for each $i \in \{ 1,\ldots,n \}$) to get
\begin{equation} \label{varphideltabound}
\varphi_{\delta}(x) \in [-1,0] \text{ for each } x \in \overline{\Omega}.
\end{equation}

Next, we aim to apply Lemma \ref{Morreyestimate} to $\varphi_{\delta}$. By (i), \eqref{memollification}, Remark \ref{Morreyspaceremark}, and Lemma \ref{mollificationproperties} we can conclude
$$\sum_{i,j=1}^{n} \| a^{ij} \|_{C^{0,\alpha}(\overline{\Omega})} + \sum_{i=1}^{n} \left( \| b^{i}_{\delta} \|_{C^{0,\alpha}(\overline{\Omega})} + \| c^{i}_{\delta} \|_{L^{1,\alpha}(\Omega)} \right) + \| d_{\delta} \|_{L^{1,\alpha}(\Omega)} \leq J$$
where, with $C_{\ref{mollificationproperties}}=C_{\ref{mollificationproperties}}(n)$ by Lemma \ref{mollificationproperties}, we let
$$J = \sum_{i,j}^{n} \| a^{ij} \|_{C^{0,\alpha}(\overline{\Omega})} + \sum_{i=1}^{n} \left( C_{\ref{mollificationproperties}} \| b^{i} \|_{C^{0,\alpha}(\overline{\Omega})} + \| c^{i} \|_{L^{1,\alpha}(\Omega)} \right) + 2^{n} \| d \|_{L^{1,\alpha}(\Omega)}.$$
Thus, by \eqref{varphidelta}, Lemma \ref{Morreyestimate}, and \eqref{varphideltabound} we conclude
$$\| \varphi_{\delta} \|_{C^{1,\alpha}(\overline{\Omega})} \leq C_{\ref{Morreyestimate}} \| \varphi_{\delta} \|_{L^{1}(\Omega)} \leq C_{\ref{Morreyestimate}} \omega_{n}$$
where $C_{\ref{Morreyestimate}}=C_{\ref{Morreyestimate}}(n,\lambda,J,\alpha) \in (0,\infty)$ does not depend on $\delta.$

We conclude there is $\varphi \in C^{1,\alpha}(\overline{\Omega})$ so that $\varphi_{\delta} \rightarrow \varphi$ in the $C^{1}(\overline{\Omega})$-norm as $\delta \searrow 0.$ Lemma \ref{mollificationproperties}(i) and \eqref{varphidelta},\eqref{varphideltabound} imply $\varphi$ is the desired solution.
\end{proof}

\section{The Hopf boundary point lemma} \label{thehopfboundarypointlemma}

We are now ready to state and prove our main result.

\begin{thrm}[Hopf boundary point lemma] \label{hopflemma} 
Suppose $q>n$ and $\lambda \in (0,\infty).$ With $\alpha = 1-\frac{n}{q},$ suppose
\begin{enumerate}
 \item[(i)] $a^{ij},b^{i} \in C^{0,\alpha}(\overline{B_{1}(0)}),$ $c^{i} \in L^{q}(B_{1}(0))$ for $i,j \in \{ 1,\ldots,n \}$ and \\ $d \in L^{\frac{q}{2}}(B_{1}(0)) \cap L^{1,\alpha}(B_{1}(0)),$
 \item[(ii)] $\{ a^{ij} \}_{i,j=1}^{n}$ are uniformly elliptic over $B_{1}(0)$ with respect to $\lambda,$
 \item[(iii)] $\{ b^{i} \}_{i=1}^{n},d$ are weakly non-positive over $B_{1}(0),$
 \item[(iv)] $a^{ij}(-e_{n})=a^{ji}(-e_{n})$ for each $i,j \in \{ 1,\ldots,n \}.$
\end{enumerate}
If $\cu \in C(\overline{B_{1}(0)}) \cap W^{1,2}(B_{1}(0))$ is a weak solution over $B_{1}(0)$ to the equation
$$\sum_{i,j=1}^{n} D_{i} \left( a^{ij} D_{j}\cu+b^{i}\cu \right) + \sum_{i=1}^{n} c^{i} D_{i}\cu + d\cu \leq 0$$
and $\cu(x)>\cu(-e_{n})=0$ for all $x \in B_{1}(0),$ then $\liminf_{h \searrow 0} \frac{\cu((h-1)e_{n})}{h} > 0.$
\end{thrm}

\begin{proof} 
Set $\Omega = B_{1}(0) \setminus \overline{B_{\frac{1}{2}}(0)},$ and using (i) define $J,k,K \in (0,\infty)$ by
\begin{equation} \label{constants}
\begin{aligned}
J = & \sum_{i,j}^{n} \| a^{ij} \|_{C^{0,\alpha}(\overline{\Omega})} + \sum_{i=1}^{n} \left( \| b^{i} \|_{C^{0,\alpha}(\overline{\Omega})} + \| c^{i} \|_{L^{1,\alpha}(\Omega)} \right) + \| d \|_{L^{1,\alpha}(\Omega)}, \\
k = & \sum_{i=1}^{n} \left( \| b^{i} \|_{L^{q}(\Omega)} + \| c^{i} \|_{L^{q}(\Omega)} \right) + \| d \|_{L^{\frac{q}{2}}(\Omega)}, \\
K = & \sum_{i,j}^{n} \| a^{ij} \|_{C^{0,\alpha}(\overline{\Omega})} + \sum_{i=1}^{n} \left( \| b^{i} \|_{L^{q}(\Omega)} + \| c^{i} \|_{L^{q}(\Omega)} \right) + \| d \|_{L^{\frac{q}{2}}(\Omega)}; 
\end{aligned}
\end{equation}
note that we used Remark \ref{Morreyspaceremark} to conclude $c^{i} \in L^{1,\alpha}(\Omega).$ The proof now proceeds through five major steps.

{\bf Step 1:} Freezing at the origin and the barrier $\varphi.$

\medskip

Consider the operator $L$ given by
$$Lu = \sum_{i,j=1}^{n} a^{ij}(-e_{n}) D_{ij} u \text{ over } \Omega.$$
Then (i),(ii),(iv) imply $L$ satisfies (6.1),(6.2) of \cite{GT83}. Applying Theorems 6.14 of \cite{GT83} 
$$\text{(with $a^{ij},b^{i},c,f$ replaced respectively by $a^{ij}(-e_{n}),0,0,0$)}$$
over $\Omega = B_{1}(0) \setminus \overline{B_{\frac{1}{2}}(0)}$ with operator $L,$ we conclude the Dirichlet problem
$$\begin{array}{lr}
Lu=0 \text{ in } \Omega, & u = \left\{ 
\begin{array}{cl} 
0 & \text{over } \partial B_{1}(0) \\ 
-1 & \text{over } \partial B_{\frac{1}{2}}(0)
\end{array} \right.
\end{array}$$
has a unique solution lying in $C^{2,\alpha}(\overline{\Omega}).$ If we let $\varphi \in C^{2,\alpha}(\overline{\Omega})$ be this unique solution, we conclude that $\varphi$ satisfies
\begin{equation} \label{varphi}
\begin{aligned}
& \sum_{i,j=1}^{n} a_{ij}(-e_{n}) D_{ij} \varphi = 0 \text{ over } \Omega \\
& \text{with } \varphi(x) = \left\{ 
\begin{array}{cl} 
0 & \text{for } x \in \partial B_{1}(0) \\ 
-1 & \text{for } x \in \partial B_{\frac{1}{2}}(0).
\end{array} \right.
\end{aligned}
\end{equation}
Using again (i),(ii),(iv) we see $L$ satisfies (3.1),(3.2),(3.3) of \cite{GT83}. Thus, the strong maximum principle, see Theorem 3.5 of \cite{GT83}
$$\text{(with $a^{ij},b^{i},c,f$ replaced respectively by $a^{ij}(-e_{n}),0,0,0$),}$$
implies $\varphi(x) \in (-1,0)$ for all $x \in \Omega$. This now means the classical Hopf boundary point lemma, see Lemma 3.4 of \cite{GT83}
$$\text{(with $x_{0},a^{ij},b^{i},c,f$ replaced respectively by $-e_{n},a^{ij}(-e_{n}),0,0,0$),}$$
implies $D_{n} \varphi(-e_{n}) < 0.$

{\bf Step 2:} Scaling and the barrier $\varphi_{\epsilon}.$

\medskip

For each $\epsilon \in (0,\frac{1}{4})$ and $i,j \in \{ 1,\ldots,n\}$ define (over $\overline{\Omega}$ or $\Omega$)
\begin{equation} \label{scaling}
\begin{aligned} 
\cu_{\epsilon}(x) &= \cu(\epsilon (x+e_{n})-e_{n}) & & \\
a^{ij}_{\epsilon}(x) &= a^{ij}(\epsilon (x+e_{n})-e_{n}) & b^{i}_{\epsilon}(x) &= \epsilon b^{i}(\epsilon (x+e_{n})-e_{n}) \\
c^{i}_{\epsilon}(x) &= \epsilon c^{i}(\epsilon (x+e_{n})-e_{n}) & d_{\epsilon}(x) &= \epsilon^{2} d(\epsilon (x+e_{n})-e_{n}).
\end{aligned}
\end{equation}
Observe that (using the change of variables $y=\epsilon x$ and Definition \ref{Morreyspace})
\begin{equation} \label{scalingnorms}
\begin{aligned}
\| a^{ij}_{\epsilon} \|_{C^{0,\alpha}(\overline{\Omega})} & \leq \| a^{ij} \|_{C^{0,\alpha}(\overline{\Omega})}, & & \\
\| b^{i}_{\epsilon} \|_{C^{0,\alpha}(\overline{\Omega})} & \leq \epsilon \| b^{i} \|_{C^{0,\alpha}(\overline{\Omega})}, & \| b^{i}_{\epsilon} \|_{L^{q}(\Omega)} & \leq \epsilon^{\alpha} \| b^{i} \|_{L^{q}(\Omega)}, \\
\| c^{i}_{\epsilon} \|_{L^{1,\alpha}(\Omega)} & \leq \epsilon^{\alpha} \| c^{i} \|_{L^{1,\alpha}(\Omega)}, & \| c^{i}_{\epsilon} \|_{L^{q}(\Omega)} & \leq \epsilon^{\alpha} \| c^{i} \|_{L^{q}(\Omega)}, \\
\| d_{\epsilon} \|_{L^{1,\alpha}(\Omega)} & \leq \epsilon^{1+\alpha} \| d \|_{L^{1,\alpha}(\Omega)}, & \| d_{\epsilon} \|_{L^{\frac{q}{2}}(\Omega)} & \leq \epsilon^{2\alpha} \| d \|_{L^{\frac{q}{2}}(\Omega)},
\end{aligned}
\end{equation}
for each $i,j \in \{ 1,\ldots,n \},$ since $\epsilon \in (0,\frac{1}{4})$ implies $\{ \epsilon (x+e_{n})-e_{n}: x \in \Omega \} \subset \Omega.$ We as well have by (ii),(iii)
\begin{equation} \label{scalingassumptions}
\begin{aligned}
& \text{$\{ a^{ij}_{\epsilon} \}$ are uniformly elliptic over $\Omega$ with respect to $\lambda,$} \\
& \text{$\{ b^{i}_{\epsilon} \}_{i=1}^{n},d_{\epsilon}$ are weakly non-positive over $\Omega.$}
\end{aligned}
\end{equation}
Using \eqref{scalingnorms},\eqref{scalingassumptions} we conclude by Lemma \ref{Morreyexistence} that for each $\epsilon \in (0,\frac{1}{4})$ there is $\varphi_{\epsilon} \in C^{1,\alpha}(\overline{\Omega})$ which is a weak solution over $\Omega$ of the equation
\begin{equation} \label{varphiepsilon}
\begin{aligned}
& \sum_{i,j=1}^{n} D_{i} \left( a^{ij}_{\epsilon} D_{j} \varphi_{\epsilon} + b^{i}_{\epsilon} \varphi_{\epsilon} \right) + \sum_{i=1}^{n} c^{i}_{\epsilon} D_{i} \varphi_{\epsilon} + d_{\epsilon} \varphi_{\epsilon} = 0 \\
& \text{with } \varphi_{\epsilon}|_{\partial \Omega} = \varphi|_{\partial \Omega} \text{ and } \varphi_{\epsilon} \in [-1,0] \text{ for each } x \in \overline{\Omega}.
\end{aligned}
\end{equation}

{\bf Step 3:} Comparing $\varphi$ and $\varphi_{\epsilon}.$

\medskip

Define the functions
\begin{equation} \label{gf}
g_{\epsilon} = - \sum_{i=1}^{n} c^{i}_{\epsilon} D_{i} \varphi - d_{\epsilon} \varphi \text{ and }
f^{i}_{\epsilon} = - \sum_{j=1}^{n} (a^{ij}_{\epsilon} - a^{ij}_{\epsilon}(-e_{n})) D_{j} \varphi - b^{i}_{\epsilon} \varphi  
\end{equation}
for $i \in \{ 1,\ldots,n \}.$ Then \eqref{varphi},\eqref{varphiepsilon},\eqref{gf} imply that $\psi_{\epsilon} = \varphi_{\epsilon}-\varphi \in C^{1,\alpha}(\overline{\Omega})$ is a weak solution over $\Omega$ of the equation
\begin{equation} \label{psiepsilon}
\begin{aligned}
& \sum_{i,j=1}^{n} D_{i}( a^{ij}_{\epsilon} D_{j} \psi_{\epsilon} + b^{i}_{\epsilon} \psi_{\epsilon}) + \sum_{i=1}^{n} c^{i}_{\epsilon} D_{i} \psi_{\epsilon} + d_{\epsilon} \psi_{\epsilon} = g_{\epsilon} + \sum_{i=1}^{n} D_{i} f^{i}_{\epsilon} \\
& \text{with } \psi_{\epsilon}|_{\partial \Omega}=0.
\end{aligned}
\end{equation}

We wish to apply Lemma \ref{Morreyestimate} to $\psi_{\epsilon}.$ Before we do so, we will use Lemma \ref{weakmaximumprinciple} to estimate $\| \psi_{\epsilon} \|_{L^{1}(\Omega)}.$ For this, we make the following three computations.

First, by \eqref{scalingnorms}, $\epsilon \in (0,\frac{1}{4}),$ $\alpha = 1-\frac{n}{q}>0,$ and with $k$ as in \eqref{constants}
$$\sum_{i=1}^{n} \left( \| b^{i}_{\epsilon} \|_{L^{q}(\Omega)} + \| c^{i}_{\epsilon} \|_{L^{q}(\Omega)} \right) + \| d_{\epsilon} \|_{L^{\frac{q}{2}}(\Omega)} \leq k.$$

Second, using \eqref{gf}, $\varphi \in C^{2,\alpha}(\overline{\Omega})$ by \eqref{varphi}, H$\ddot{\text{o}}$lder's inequality, and \eqref{scalingnorms}
$$\begin{aligned}
\| g_{\epsilon} \|_{L^{\frac{q}{2}}(\Omega)} & \leq \| \varphi \|_{C^{1}(\Omega)} \left( \sum_{i=1}^{n} \| c^{i}_{\epsilon} \|_{L^{\frac{q}{2}}(\Omega)} + \| d_{\epsilon} \|_{L^{\frac{q}{2}}(\Omega)} \right) \\
& \leq \| \varphi \|_{C^{1}(\Omega)} \left( \sum_{i=1}^{n} \epsilon^{\alpha} \omega_{n}^{\frac{1}{q}} \| c^{i} \|_{L^{q}(\Omega)} + \epsilon^{2\alpha} \| d \|_{L^{\frac{q}{2}}(\Omega)} \right). \\
\end{aligned}$$

Third, we similarly compute for each $i \in \{ 1,\ldots,n \}$ using \eqref{gf},\eqref{scaling},\eqref{scalingnorms} 
$$\begin{aligned}
\| f^{i}_{\epsilon} \|_{L^{q}(\Omega)} & \leq \| \varphi \|_{C^{1}(\Omega)} \left( \sum_{j=1}^{n} \omega_{n}^{\frac{1}{q}} \| a^{ij}_{\epsilon}-a^{ij}_{\epsilon}(-e_{n}) \|_{C(\Omega)} + \| b^{i}_{\epsilon} \|_{L^{q}(\Omega)} \right) \\
& \leq \| \varphi \|_{C^{1}(\Omega)} \left( \sum_{j=1}^{n} \epsilon^{\alpha} 2^{\alpha}  \omega_{n}^{\frac{1}{q}} \| a^{ij} \|_{C^{0,\alpha}(\overline{\Omega})} + \epsilon^{\alpha} \| b^{i} \|_{L^{q}(\Omega)} \right).
\end{aligned}$$

These three computations together with \eqref{psiepsilon} imply that we can apply Lemma \ref{weakmaximumprinciple} (with $\cu=\psi_{\epsilon},-\psi_{\epsilon}$) to conclude
$$\begin{aligned}
\sup_{\Omega} |\psi_{\epsilon}| \leq & C_{\ref{weakmaximumprinciple}} \left( \| g_{\epsilon} \|_{L^{\frac{q}{2}}(\Omega)} + \sum_{i=1}^{n} \| f^{i}_{\epsilon} \|_{L^{q}(\Omega)} \right) \\
\leq & C_{\ref{weakmaximumprinciple}} \| \varphi \|_{C^{1}(\Omega)} \left( \sum_{i=1}^{n} \epsilon^{\alpha} \omega_{n}^{\frac{1}{q}} \| c^{i} \|_{L^{q}(\Omega)} + \epsilon^{2\alpha} \| d \|_{L^{\frac{q}{2}}(\Omega)} \right) \\
& + C_{\ref{weakmaximumprinciple}} \| \varphi \|_{C^{1}(\Omega)} \left( \sum_{i,j=1}^{n} \epsilon^{\alpha} 2^{\alpha}  \omega_{n}^{\frac{1}{q}} \| a^{ij} \|_{C^{0,\alpha}(\overline{\Omega})} + \sum_{i=1}^{n} \epsilon^{\alpha} \| b^{i} \|_{L^{q}(\Omega)} \right) \\
\leq & \epsilon^{\alpha} C_{\ref{weakmaximumprinciple}} \max \{ 2^{\alpha} \omega_{n}^{\frac{1}{q}},1 \} \| \varphi \|_{C^{1}(\Omega)} K
\end{aligned}$$
where $C_{\ref{weakmaximumprinciple}} = C_{\ref{weakmaximumprinciple}}(n,q,\lambda,k)$ and $k,K$ as in \eqref{constants} do not depend on $\epsilon.$ Thus
\begin{equation} \label{psiepsilonL1}
\| \psi_{\epsilon} \|_{L^{1}(\Omega)} \leq \epsilon^{\alpha} C_{\ref{weakmaximumprinciple}} \max \{ 2^{\alpha} \omega_{n}^{1+\frac{1}{q}},\omega_{n} \} \| \varphi \|_{C^{1}(\Omega)} K.
\end{equation}

Now we shall use Lemma \ref{Morreyestimate}. For this we make three computations.

First, using \eqref{scalingnorms}, $\epsilon \in (0,\frac{1}{4}),$ $\alpha = 1-\frac{n}{q}>0,$ and with $J$ as in \eqref{constants}
$$\sum_{i,j=1}^{n} \| a^{ij}_{\epsilon} \|_{C^{0,\alpha}(\overline{\Omega})} + \sum_{i=1}^{n} \left( \| b^{i}_{\epsilon} \|_{C^{0,\alpha}(\overline{\Omega})} + \| c^{i}_{\epsilon} \|_{L^{1,\alpha}(\Omega)} \right) + \| d_{\epsilon} \|_{L^{1,\alpha}(\Omega)} \leq J.$$

Second, we compute using Definition \ref{Morreyspace} and \eqref{gf},\eqref{scalingnorms}
$$\begin{aligned}
\| g_{\epsilon} \|_{L^{1,\alpha}(\Omega)} \leq & \| \varphi \|_{C^{1}(\Omega)} \left( \sum_{i=1}^{n} \| c^{i}_{\epsilon} \|_{L^{1,\alpha}(\Omega)} + \| d_{\epsilon} \|_{L^{1,\alpha}(\Omega)} \right) \\
\leq & \| \varphi \|_{C^{1}(\Omega)} \left( \sum_{i=1}^{n} \epsilon^{\alpha} \| c^{i} \|_{L^{1,\alpha}(\Omega)} + \epsilon^{2 \alpha} \| d \|_{L^{1,\alpha}(\Omega)} \right).
\end{aligned}$$

Third, we compute for each $i \in \{ 1,\ldots,n \}$ using \eqref{scaling},\eqref{scalingnorms} 
$$\begin{aligned}
\| f^{i}_{\epsilon} \|_{C^{0,\alpha}(\overline{\Omega})} \leq & \| \varphi \|_{C^{1,\alpha}(\overline{\Omega})} \left( \sum_{j=1}^{n} \| a^{ij}_{\epsilon}-a^{ij}_{\epsilon}(-e_{n})\|_{C^{0,\alpha}(\overline{\Omega})} + \| b^{i}_{\epsilon} \|_{C^{0,\alpha}(\overline{\Omega})} \right) \\
\leq & \| \varphi \|_{C^{1,\alpha}(\overline{\Omega})} \left( \sum_{j=1}^{n} \epsilon^{\alpha} \| a^{ij} \|_{C^{0,\alpha}(\overline{\Omega})} + \epsilon^{\alpha} \| b^{i} \|_{C^{0,\alpha}(\overline{\Omega})} \right).
\end{aligned}$$

These three computations, \eqref{psiepsilon}, Lemma \ref{Morreyestimate}, and \eqref{psiepsilonL1} imply
$$\begin{aligned}
\| \psi_{\epsilon} \|_{C^{1,\alpha}(\overline{\Omega})} \leq & C_{\ref{Morreyestimate}} \left( \| \psi_{\epsilon} \|_{L^{1}(\Omega)} + \| g_{\epsilon} \|_{L^{1,\alpha}(\Omega)} + \sum_{i=1}^{n} \| f^{i}_{\epsilon} \|_{C^{0,\alpha}(\overline{\Omega})} \right) \\
\leq & \epsilon^{\alpha} C_{\ref{Morreyestimate}} C_{\ref{weakmaximumprinciple}} \max \{ 2^{\alpha} \omega_{n}^{1+\frac{1}{q}},\omega_{n} \} \| \varphi \|_{C^{1}(\Omega)} K \\
& + C_{\ref{Morreyestimate}} \| \varphi \|_{C^{1}(\Omega)} \left( \sum_{i=1}^{n} \epsilon^{\alpha} \| c^{i} \|_{L^{1,\alpha}(\Omega)} + \epsilon^{2 \alpha} \| d \|_{L^{1,\alpha}(\Omega)} \right) \\
& + C_{\ref{Morreyestimate}} \| \varphi \|_{C^{1,\alpha}(\overline{\Omega})} \left( \sum_{i,j=1}^{n} \epsilon^{\alpha} \| a^{ij} \|_{C^{0,\alpha}(\overline{\Omega})} + \sum_{i=1}^{n} \epsilon^{\alpha} \| b^{i} \|_{C^{0,\alpha}(\overline{\Omega})} \right) \\
\leq & \epsilon^{\alpha} C_{\ref{Morreyestimate}} C_{\ref{weakmaximumprinciple}} \max \{ 2^{\alpha} \omega_{n}^{1+\frac{1}{q}},\omega_{n} \} \| \varphi \|_{C^{1}(\Omega)} K \\
& + \epsilon^{\alpha} C_{\ref{Morreyestimate}} \| \varphi \|_{C^{1,\alpha}(\overline{\Omega})} J
\end{aligned}$$
where $C_{\ref{Morreyestimate}} = C_{\ref{Morreyestimate}}(n,\lambda,J,\alpha)$ and $J$ as in \eqref{constants} do not depend on $\epsilon.$ Recalling that $C_{\ref{weakmaximumprinciple}}=C_{\ref{weakmaximumprinciple}}(n,q,\lambda,k)$ and $k,K$ as in \eqref{constants} do not depend on $\epsilon,$ then
\begin{equation} \label{psiepsilonderivative}
\lim_{\epsilon \rightarrow 0} |D_{n}\varphi_{\epsilon}(-e_{n})-D_{n}\varphi(-e_{n})| \leq \lim_{\epsilon \rightarrow 0} \| \psi_{\epsilon} \|_{C^{1,\alpha}(\overline{\Omega})} = 0.
\end{equation}

{\bf Step 4:} Fixing $\epsilon$ and comparing $\cu_{\epsilon}$ and $\varphi_{\epsilon}.$

\medskip

By {\bf Step 1} and \eqref{psiepsilonderivative}, we can fix $\epsilon \in (0,\frac{1}{4})$ so that
\begin{equation} \label{varphiepsilonderivative}
D_{n} \varphi_{\epsilon}(-e_{n}) < 0.
\end{equation}
Recalling $\cu \in C(\overline{B_{1}(0)})$ with $\cu(x)>\cu(-e_{n})=0$ for $x \in B_{1}(0),$ we can define $\hat{\cu}_{\epsilon} \in C(\overline{\Omega}) \cap W^{1,2}(\Omega)$ by
\begin{equation} \label{uepsilon}
\hat{\cu}_{\epsilon} = (\cu_{\epsilon}+\theta_{\epsilon} \varphi_{\epsilon}) \text{ with } \theta_{\epsilon} = \inf_{\partial B_{\frac{1}{2}}(0)} \cu_{\epsilon}>0;
\end{equation}

Observe by \eqref{scaling} that $\cu_{\epsilon}$ is a weak solution over $\Omega$ of the equation
$$\sum_{i,j=1}^{n} D_{i} \left( a^{ij}_{\epsilon} D_{j}\cu_{\epsilon} + b^{i}_{\epsilon} \cu_{\epsilon} \right) + \sum_{i=1}^{n} c^{i}_{\epsilon} D_{i}\cu_{\epsilon} + d_{\epsilon}\cu_{\epsilon} \leq 0,$$
Then \eqref{varphi},\eqref{varphiepsilon},\eqref{uepsilon} imply $\hat{\cu}_{\epsilon}$ is a weak solution over $\Omega$ of the equation
$$\sum_{i,j=1}^{n} D_{i} \left( a^{ij}_{\epsilon} D_{j} \hat{\cu}_{\epsilon} +b^{i}_{\epsilon} \hat{\cu}_{\epsilon} \right) + \sum_{i=1}^{n} c^{i}_{\epsilon} D_{i} \hat{\cu}_{\epsilon} + d_{\epsilon} \hat{\cu}_{\epsilon} \leq 0 \text{ with } \hat{\cu}_{\epsilon}|_{\partial \Omega} \geq 0.$$
We conclude by Lemma \ref{weakmaximumprinciple} that $\inf_{\Omega} \hat{\cu}_{\epsilon} \geq 0.$

{\bf Step 5:} Computing the derivative of $\cu$ at the origin.

Using \eqref{scaling},\eqref{varphiepsilonderivative},\eqref{uepsilon} and $\inf_{\Omega} \hat{\cu}_{\epsilon} \geq 0$ we conclude
$$\begin{aligned}
\liminf_{h \searrow 0} \frac{\cu((h-1)e_{n})}{h} = & \liminf_{h \searrow 0} \frac{\cu_{\epsilon}((\frac{h}{\epsilon}-1)e_{n})}{h} \\
\geq & \liminf_{h \searrow 0} \frac{-\theta_{\epsilon} \varphi_{\epsilon}((\frac{h}{\epsilon}-1)e_{n})}{h} = \frac{-\theta_{\epsilon}}{\epsilon} D_{n} \varphi_{\epsilon}(-e_{n}) > 0.
\end{aligned}$$
\end{proof}

It is typical to make some remarks relaxing some of the assumptions on the coefficients in certain cases; see for example Remark 1.2(b) of \cite{S15}. We make two more similar remarks.

\begin{remark} \label{hopfremark} 
We can relax some of the assumptions of Theorem \ref{hopflemma}.
\begin{itemize}
 \item[(i)] We need not assume $\alpha = 1-\frac{n}{q},$ it merely suffices that
$$\begin{array}{ccc}
a^{ij},b^{i} \in C^{0,\alpha}(\overline{B_{1}(0)}),
&
c^{i} \in L^{q}(B_{1}(0)),
&
d \in L^{\frac{q}{2}}(B_{1}(0)) \cap L^{1,\alpha}(B_{1}(0))
\end{array}$$
with $q>n$ and general $\alpha \in (0,1).$ 
 \item[(ii)] We can more generally assume $\cu(-e_{n}) \leq 0.$ We can see this by setting $\hat{\cu}(x)=\cu(x)-\cu(-e_{n})$ for $x \in \overline{\Omega},$ and noting that for $\zeta \in C^{1}_{c}(\Omega;[0,\infty))$
$$\begin{aligned}
\int & \sum_{i,j=1}^{n} a^{ij}D_{j}\hat{\cu} D_{i} \zeta + \sum_{i=1}^{n} \left( b^{i} \hat{\cu} D_{i} \zeta - c^{i} (D_{i} \hat{\cu}) \zeta \right) - d \hat{\cu} \zeta \dx \\
= & \int \sum_{i,j=1}^{n} a^{ij}D_{j}\cu D_{i} \zeta + \sum_{i=1}^{n} \left( b^{i}\cu D_{i} \zeta - c^{i} (D_{i}\cu) \zeta \right) - d\cu \zeta \dx \\
& + \cu(-e_{n}) \int d \zeta - \sum_{i=1}^{n} b^{i} D_{i} \zeta \dx \geq 0
\end{aligned}$$
since $\{ b_{i} \}_{i=1}^{n},d$ are weakly non-positive over $\Omega.$
\end{itemize}
\end{remark}


\begin{thebibliography}{10} 

\bibitem{AKM17}
B. Avelin, T. Kuusi, G. Mingione,
\emph{Nonlinear Calderon-Zygmund theory in the limiting case.}
Arch. Rat. Mech. Anal.
(2017) (DOI) 10.1007/s00205-017-1171-7

\bibitem{CCAL17}
P. Cianci, G.R. Cirmi, S. D'Asero, S. Leonardi,
\emph{Morrey estimates for solutions of singular quadratic non linear equations,}
Ann. Mat. Pura e Appl.
{\bf 196}(2017), no.~5, 1739--1758.

\bibitem{CAL17}
G. R. Cirmi, S. D'Asero, S. Leonardi, 
\emph{Gradient estimate for solutions of a class of nonlinear elliptic equations below the duality exponent,}
Math. Method Appl. Sci.
(2017) (DOI) 10.1002/mma.4609

\bibitem{CL06}
G.R. Cirmi, S. Leonardi,
\emph{Regularity results for the gradient of solutions of linear elliptic equations with $L^{1,\lambda}$ data,}
Ann. Mat. Pura e Appl.
{\bf 185}(2006), no.~4, 537--553.

\bibitem{CL10}
G.R. Cirmi, S. Leonardi, 
\emph{Higher differentiability for solutions of linear elliptic systems with measure data,}
Discret. Cont. Dyn.-A
{\bf 26}(2010), no.~1, 89--104.

\bibitem{CL14}
G.R. Cirmi, S. Leonardi, 
\emph{Higher differentiability for the solutions of nonlinear elliptic systems with lower order terms and $L^{1,\theta}$-data,}
Ann. Mat. Pura e Appl.
{\bf 193}(2014), no.~1, 115--131.

\bibitem{CLS08}
G.R. Cirmi, S. Leonardi, J. Star\'a,
\emph{Regularity results for the gradient of solutions of a class of linear elliptic systems with $L^{1,\lambda}$ data,}
Nonlinear Anal.-Theor.
{\bf 68}(2008), no.~12, 2609--3624.

\bibitem{GT83} 
D. Gilbarg and N.S. Trudinger,
\emph{Elliptic partial differential equations of second order.} 
Second edition. Springer-Verlag, Berlin-Heidelberg-New York, 1983.

\bibitem{HS79}
R. Hardt and L. Simon,
{Boundary regularity and embedded solutions for the oriented Plateau problem.} 
Ann. Math. 
{\bf 110}(1979), no.~3, 439--486.

\bibitem{H52}
E. Hopf,
{A remark on linear elliptic differential equations of second order.} 
Proc. Amer. Math. Soc.
{\bf 3}(1952), no.~5, 80--85.

\bibitem{KM10}
J. Kristensen, G. Mingione, 
\emph{Boundary regularity in variational problems,}
Arch. Rat. Mech. Anal.
{\bf 198}(2010), no.~2, 369-455.

\bibitem{L11}
S. Leonardi, 
\emph{Gradient estimates below duality exponent for a class of linear elliptic systems,}
Nonlinear Differ. Equ. Appl.,
{\bf 18}(2011), no.~3, 237--254.

\bibitem{L14}
S. Leonardi, 
\emph{Fractional differentiability for solutions of a class of parabolic systems with $L^{1,\theta}$-data,} Nonlinear Anal.-Theor.
{\bf 95}(2014), 530--542.

\bibitem{LS10}
S. Leonardi, J. Star\'a,
\emph{Regularity results for the gradient of solutions of linear elliptic systems with VMO-coefficients and $L^{1,\lambda}$ data,}
Forum Math.
{\bf 22}(2010), no.~5, 913--940.

\bibitem{LS11}
S. Leonardi, J. Star\'a,
\emph{Regularity up to the boundary for the gradient of solutions of linear elliptic systems with VMO coefficients and $L^{1,\lambda}$,}
Complex Var. Elliptic
{\bf 56}(2011), no.~12, 1086--1098.

\bibitem{LS12}
S. Leonardi, J. Star\'a,
\emph{Regularity results for solutions of a class of parabolic systems with measure data,}
Nonlinear Anal.-Theor.
{\bf 75}(2012), no.~4, 2069--2089.

\bibitem{LS15}
S. Leonardi, J. Star\'a,
\emph{Higher differentiability for solutions of a class of parabolic systems with $L^{1,\theta}$-data,} 
Q. J. Math.
{\bf 66}(2015), no.~1, 659--676.

\bibitem{M07}
G. Mingione,
\emph{The Calderon-Zygmund theory for elliptic problems with measure data.}
Ann. Scuola Norm. Sup. Pisa Cl. Sci. 
{\bf 5}(2007), no.~4, 195--261.

\bibitem{M66} 
C.B. Morrey Jr,
\emph{ Multiple integrals in the calculus of variations.} 
Springer-Verlag, Berlin-Heidelberg-New York, 1966.

\bibitem{S15}
J.C. Sabina De Lis,
{Hopf maximum principle revisited.} 
Electron. J. Differ. Eq.
{\bf 115}(2015), 1--9.

\end{thebibliography}
\end{document}